\documentclass[12pt]{amsart}
\usepackage{amsmath,amsthm,amsfonts,amssymb,mathrsfs}
\usepackage[all]{xy}
\usepackage[T1]{fontenc}
\usepackage{hyperref}

\newcommand{\C}{\mathbb{C}}
\newcommand{\Z}{\mathbb{Z}}

\newcommand{\Q}{\mathbb{Q}}
\newcommand{\F}{\mathbb{F}}
\newcommand{\A}{\mathbb{A}}

\newcommand{\cG}{\mathcal{G}}
\newcommand{\cL}{\mathcal{L}}

\newcommand{\g}{\mathfrak{g}}

\newcommand{\Gal}{\operatorname{Gal}}

\newcommand{\Spec}{\operatorname{Spec}\,}

\newcommand{\Ad}{\operatorname{Ad}}
\newcommand{\ad}{\operatorname{ad}}

\newcommand{\pr}{\operatorname{pr}}
\renewcommand{\sc}{\operatorname{sc}}
\renewcommand{\ss}{\operatorname{ss}}
\newcommand{\der}{\operatorname{der}}

\newcommand{\Aut}{\operatorname{Aut}}

\newcommand{\GL}{\mathrm{GL}}

\newcommand{\SL}{\mathrm{SL}}

\newtheorem{thm}{Theorem}[section]

\newtheorem{cor}[thm]{Corollary}

\newtheorem{prop}[thm]{Proposition}
\newtheorem{lemma}[thm]{Lemma}

\newtheorem{conj}[thm]{Conjecture}

\newtheorem{remark}[thm]{Remark}

\begin{document}
\title{Adelic openness without the Mumford-Tate conjecture}

\thanks{$^*$Chun Yin Hui is supported by the National Research Fund, Luxembourg, and cofunded under the Marie Curie Actions of the European Commission (FP7-COFUND)}
\thanks{$^{**}$Michael Larsen was partially supported by NSF grant DMS-1101424 and a Simons Fellowship.}
\maketitle

\begin{center}
Chun Yin Hui$^*$ and Michael Larsen$^{**}$
\end{center}

\vspace{.1in}

\begin{center}
$^*$Mathematics Research Unit, University of Luxembourg,\\
6 rue Richard Coudenhove-Kalergi, L-1359 Luxembourg\\
Email: \url{pslnfq@gmail.com}
\end{center}

\vspace{.1in}

\begin{center}
$^{**}$Department of Mathematics, Indiana University, \\
Bloomington, IN 47405, U.S.A.\\
Email: \url{mjlarsen@indiana.edu}
\end{center}

\vspace{.1in}
\begin{abstract}
Let $X$ be a non-singular projective variety over a number field $K$, $i$ a non-negative integer, and $V_{\A}$, 
the $i$th \'etale cohomology of $\bar X:=X\times_K\bar K$ with coefficients in the ring of finite adeles $\A_f$ over $\Q$. 
Assuming the \emph{Mumford-Tate conjecture}, we formulate a conjecture (Conj. \ref{MT}) 
describing the largeness of the image of the absolute Galois group $G_K$ in $H(\A_f)$ under the adelic Galois representation 
$\Phi_{\A}: G_K \to \Aut(V_{\A}),$ 
where $H$ is the corresponding Mumford-Tate group. 
The motivating example is a celebrated theorem of Serre, which asserts that if $X$ is an elliptic curve without complex multiplication over $\bar K$ and $i=1$, then $\Phi_{\A}(G_K)$ is an open subgroup of $\GL_2(\hat \Z)\subset \GL_2(\A_f)=H(\A_f)$. We state and in some cases prove a weaker conjecture (Conj. \ref{main}) which does not require the Mumford-Tate conjecture but which, together with it, imply Conjecture \ref{MT}. 
We also relate our conjectures to Serre's conjectures on \emph{maximal motives} \cite{Serre-GM}.

\end{abstract}

\newpage
\section{Introduction}

Let $X$ be a non-singular projective variety over a number field $K$, $i$ a non-negative integer,  
and $V_{\A} := H^i(\bar X,\A_f)\cong \A_f^n$, the \'etale cohomology of $\bar X := X\times_K\bar K$ with coefficients in the ring $\A_f := \hat \Z\otimes_{\Z}\Q$ 
of finite adeles over $\Q$.  We would like to formulate a conjecture describing the image of the Galois group $G_K := \Gal(\bar K/K)$
under the adelic Galois representation $\Phi_{\A}\colon G_K\to \Aut(V_{\A})\cong \GL_n(\A_f)$.
The motivating example is Serre's theorem \cite{Serre-IM}, which asserts that
if $X$ is an elliptic curve without complex multiplication  over $\bar K$ and $i=1$, then $\Phi_{\A}(G_K)$ is an open subgroup of $\GL_2(\hat \Z)\subset \GL_2(\A_f)$.

Let $V_\ell = H^i(\bar X,\Q_\ell)$.  We suppose henceforth that $K$ is chosen large enough that the $\ell$-adic Galois representation
$V_\ell$ is \emph{connected}.  We recall that this means that the Zariski closure $G_\ell$ of the image $\Phi_\ell(G_K)$ in $\Aut(V_\ell)$ is
connected; this condition does not depend on $\ell$ (see \cite{Serre-KR}).
The Mumford-Tate conjecture asserts that there exists a
connected, reductive algebraic group $H$ over $\Q$, 
the Mumford-Tate group, and for each $\ell$ an isomorphism $G_\ell \tilde\to H\times\Q_\ell$ arising from the comparison 
between $\ell$-adic \'etale cohomology and ordinary cohomology.  
The naive form of ``adelic openness'' would be the claim that via these isomorphisms, $\Phi_{\A}(G_K)$
is in fact an \emph{open} subgroup of $H(\A_f)$.  This is known not to hold in general.
In fact, Serre has conjectured that $\Phi_{\A}(G_K)$ is open if and only if
$\Phi_{\A}$ comes from a \emph{maximal motive} \cite[$\mathsection11$]{Serre-GM}.

To modify the conjecture in order to cover the non-maximal case
one must work with the simply connected semisimple part of $H$.
On the $\ell$-adic side, this works as follows.
By the Mumford-Tate conjecture, the $G_\ell$ are all reductive.
Let $G^{\der}_\ell$ denote the derived group of $G_\ell$, $G^{\sc}_\ell$ the universal covering group of $G^{\der}_\ell$, 
and $\pi_\ell\colon G^{\sc}_\ell\to G^{\der}_\ell$ the covering isogeny.
The commutator map $G_\ell\times G_\ell\to G_\ell$ lifts to a morphism $\kappa_\ell\colon G_\ell\times G_\ell\to G^{\sc}_\ell$, and likewise,
the commutator map on $H$ lifts to $\kappa\colon H\times H\to H^{\sc}$.
As $\kappa$ and $\kappa_\ell$ are not homomorphisms, we work with the groups
generated by the images rather than the images themselves.
Since the fields defined by $\ker\Phi_\ell$ are almost linearly disjoint
(see \cite{Serre-AI}) and conjecturally  (see Remark \ref{rem-con}), $\kappa_\ell(\Phi_\ell(G_K),\Phi_\ell(G_K))$ 
generates a hyperspecial maximal compact subgroup
of $G^{\sc}_\ell(\Q_\ell)$ for all $\ell\gg 1$, we make the following conjecture:

\begin{conj}
\label{weak-MT}
With notations as above, $\kappa(\Phi_\A(G_K),\Phi_\A(G_K))$ generates an open subgroup of  $H^{\sc}(\A_f)$.
\end{conj}

In the light of recent work on word maps for $p$-adic and adelic groups, we can formulate the following slightly stronger variant of this conjecture: 

\begin{conj}
\label{MT}
With notations as above, we have:
\begin{itemize}
\item[(1)]
for some positive integer $k$,
$\kappa(\Phi_\A(G_K),\Phi_\A(G_K))^k$
is an open subgroup of $H^{\sc}(\A_f)$.  
\item[(2)]For all $k\ge 2$, $\kappa(\Phi_\A(G_K),\Phi_\A(G_K))^k$
contains an open subgroup
of $H^{\sc}(\A_f)$.
\end{itemize}
\end{conj}

Note that  this conjecture is somewhat weaker than Serre's result in the special case that $X$ is a non-CM elliptic curve; it implies only that
$\Phi_{\A}(G_K)$ contains an open subgroup of $\SL_2(\hat \Z)$.  The fact that the image is actually open in $\GL_2(\hat \Z)$
depends on the special circumstance that the determinant in this case is the cyclotomic character rather than some power thereof.
Conjecture~\ref{MT} is formulated so as to avoid consideration of the multiplicative part of the Galois image.

This conjecture makes sense only if the Mumford-Tate conjecture is true.  
Our goal in this paper is to state, and in some 
cases prove, a weaker conjecture which does not require Mumford-Tate but which, together with Mumford-Tate,  implies
Conjecture~\ref{MT}.  

\begin{conj}
\label{main}With notations as above, we have:
\begin{itemize}
\item[(1)]
For some $k$,  $\kappa(\Phi_\A(G_K),\Phi_\A(G_K))^k$ is a \emph{special adelic} subgroup of $\prod_\ell G_\ell^{\sc}(\Q_\ell)$.
\item[(2)]
For all $k\ge 2$, $\kappa(\Phi_\A(G_K),\Phi_\A(G_K))^k$ contains such a subgroup.
\end{itemize}
\end{conj}

In \S2, we explain what it means for a compact subgroup of a product of $\ell$-adic groups to be \emph{adelic} or \emph{special adelic}.
In \S3, we give a criterion for a compact subgroup $\Gamma\subset  \prod_\ell G_\ell^{\sc}(\Q_\ell)$ to be special adelic.
We also develop a theory, along the lines of \cite{AGKS}, which allows us to show that a product of two factors is enough to 
generate a set containing an adelic subgroup. In \S4, we show that Conjecture ~\ref{main} holds for type A representations (in the sense of \cite{HL}).
The last section clarifies the connections between our conjectures and Serre's conjectures on maximal motives.

We would like to thank Serre for a suggestion that led to this paper.

\section{Adelic subgroups}

In this section, we consider a system of simply connected semisimple algebraic groups $G_\ell/\Q_\ell$, one for
each rational prime $\ell$.  Let $G_{\A} := \prod_\ell G_\ell(\Q_\ell)$ with the natural product topology, and let $\Gamma_{\A}\subset G_{\A}$  be a compact subgroup.
We say that $\Gamma_{\A}$ is \emph{adelic} if and only if it is commensurable to $\prod_\ell \Gamma_\ell$ for some
system $\Gamma_\ell$ of maximal compact subgroups of $G_\ell(\Q_\ell)$.
If, in addition, we can choose $\Gamma_\ell$ to be hyperspecial for all $\ell$ sufficiently large, we say that $\Gamma_{\A}$ is
\emph{special adelic}. 
Denote the set of rational primes by $\mathcal{L}$. For any subset $S\subset \cL$ of $\cL$, we denote by $\pr_S$ the projection map
$$\prod_{\ell\in \cL}G_\ell(\Q_\ell)\to\prod_{\ell\in S}G_\ell(\Q_\ell).$$
If $S = \{\ell\}$ we write $\pr_\ell$ for $\pr_S$.

\begin{lemma}\label{Goursat}
(Goursat's lemma)
Let $G$ and $G'$ be profinite groups and $H$ a closed subgroup of $G\times G'$.
Let $N'$ (resp. $N$) be the kernel of the projection $p_1\colon H\to G$ (resp. $p_2:H\to G'$).
If $p_1$ and $p_2$ are surjective, then $H$ embeds as a graph of $G/N\times G'/N'$
which induces a bicontinuous isomorphism $G/N\cong G'/N'$.
\end{lemma}

\begin{lemma}
\label{One-to-finite}
Let $S$ be a finite set of primes and $\Gamma$ be a compact subgroup of $\prod_{\ell\in S}\Gamma_\ell$. If $\pr_\ell \Gamma$ is of finite index in $\Gamma_\ell$ for all $\ell\in S$, then $\Gamma$ is of finite index in $\prod_{\ell\in S}\Gamma_\ell$.
\end{lemma}

\begin{proof} The lemma is trivial if $|S|=1$. 
Assume the statement is true when $|S|=k$.
Let $S=\{\ell_1,...,\ell_{k+1}\}$.
By the induction hypothesis, for each $1\leq i\leq k$ there exists an open subgroup $\Gamma_{\ell_i}'$
of $\Gamma_{\ell_i}$ such that the projection of $\Gamma$ to $\prod_{i=1}^k\Gamma_{\ell_i}$ contains $\prod_{i=1}^k\Gamma_{\ell_i}'$.
Let
$$\Gamma' := \Gamma\cap \bigl((\Gamma'_{\ell_1}\times\cdots\times  \Gamma'_{\ell_k})\times \Gamma_{\ell_{k+1}}\bigr).$$
Then $\Gamma'$ is of finite index in $\Gamma$ and therefore maps onto a finite index subgroup $\Gamma'_{\ell_{k+1}}$ of $\Gamma_{\ell_{k+1}}$.
Replacing $\Gamma$ with $\Gamma'$ and $\Gamma_{\ell_i}$ with $\Gamma'_{\ell_i}$,
we may assume $\Gamma$ surjects onto $\prod_{i=1}^k\Gamma_{\ell_i}$ and also onto $\Gamma_{\ell_{k+1}}$.
By Goursat's lemma, we can find normal closed subgroups $N\cong\ker(\Gamma\rightarrow \Gamma_{\ell_{k+1}})$ and $N'\cong\ker(\Gamma\rightarrow \prod_{i=1}^k\Gamma_{\ell_i})$ respectively of $\prod_{i=1}^k\Gamma_{\ell_i}$ and $\Gamma_{\ell_{k+1}}$ such that 
\begin{equation}\label{Go-eqt}
(\prod_{i=1}^k\Gamma_{\ell_i})/N\cong \Gamma_{\ell_{k+1}}/N'
\end{equation}
is a bicontinuous isomorphism. Let $N_{\ell_i}$ be the projection of $N$ to $\Gamma_{\ell_i}$ for $1\leq i\leq k$. The isomorphism
(\ref{Go-eqt}) induces a continuous surjective map 
$$\Gamma_{\ell_{k+1}}/N' \rightarrow \Gamma_{\ell_i}/N_{\ell_i}$$
for $1\leq i\leq k$. Since a continuous homomorphism from a $\ell$-adic Lie group to a $p$-adic Lie group is locally constant if $p\neq \ell$ 
 \cite[$\mathsection 3.2.2.$ Proposition 1a]{Serre-ALR} and $\Gamma_{\ell_{k+1}}/N'$ is compact, it follows that $\Gamma_{\ell_i}/N_{\ell_i}$ is finite for $i\leq i\leq k$. 
By the  induction hypothesis, $N$ is of finite index in $(\prod_{i=1}^k\Gamma_{\ell_i})$. Hence, $\Gamma_{\ell_{k+1}}/N'$ 
is finite by (\ref{Go-eqt}) and the inclusions
$$N\times N'\subset \Gamma\subset \prod_{i=1}^{k+1}\Gamma_{\ell_i}$$ are of finite index. The lemma follows by induction.
\end{proof}

\begin{lemma}\label{compare}
Let $\Gamma_\ell$ be a compact open subgroup of $G_\ell(\Q_\ell)$ for all $\ell$ and $\Gamma\subset\prod_\ell\Gamma_\ell$
a compact subgroup. Let $S$ be a finite subset of $\mathcal{L}$ such that 
$$\pr_S(\Gamma)=\prod_{\ell\in S}\Gamma_\ell\hspace{.2in} \mathrm{and}\hspace{.2in} \pr_{\mathcal{L}\backslash S}(\Gamma)=\prod_{\ell\in \mathcal{L}\backslash S}\Gamma_\ell=\prod_{\ell\notin  S}\Gamma_\ell.$$
Then $\Gamma$ is of finite index in $\prod_\ell\Gamma_\ell$.
\end{lemma}

\begin{proof}
By Goursat's lemma, there exist closed normal subgroups $N_S\subset \prod_{\ell\in S}\Gamma_\ell$ and 
$N^S\subset \prod_{\ell\notin  S}\Gamma_\ell$ such that 
\begin{equation}\label{Go-eqt2}
(\prod_{\ell\in S}\Gamma_\ell)/N_S\cong (\prod_{\ell\notin  S}\Gamma_\ell)/N^S
\end{equation}
and 
$$N_S\times N^S\subset\Gamma\subset\prod_\ell\Gamma_\ell.$$
It suffices to show $N_S\times N^S$ is of finite index in $\prod_\ell\Gamma_\ell$ or, equivalently, that both sides of the isomorphism (\ref{Go-eqt2})
are finite groups.
Again, it suffices to prove that the surjection
$$\prod_{\ell\notin S}\Gamma_\ell\to (\prod_{\ell\in S}\Gamma_\ell)/N_S$$
has finite image.

If $N_\ell$ denotes the projection to $\Gamma_\ell$ of $N_S$,
by Lemma~\ref{One-to-finite}, $N_S$ is of finite index in $\prod_{\ell\in S} N_\ell$, so it suffices to prove the image of
$$\pi\colon \prod_{\ell\notin S}\Gamma_\ell\to \prod_{\ell\in S}(\Gamma_\ell/N_\ell)$$
is finite.  Now, $\Gamma_\ell$ is $\ell$-adic analytic by definition, and it follows \cite[(3.2.3.5)]{Lazard} that $\Gamma_\ell/N_\ell$ is $\ell$-adic analytic for
all $\ell\in S$.  Therefore, each $\Gamma_\ell/N_\ell$ contains an open subgroup $U_\ell$ which is an $\ell$-valued pro-$\ell$ group and hence torsion-free \cite[(3.1.3)]{Lazard}.
It suffices to prove finiteness of the image after 
replacing   $\prod_{\ell\in S}(\Gamma_\ell/N_\ell)$ by $\prod_{\ell\in S}U_\ell$ and $\prod_{\ell\not\in S}\Gamma_\ell$ by $\pi^{-1}\prod_{\ell\in S}U_\ell$.

As $\pi^{-1}\prod_{\ell\in S}U_\ell$ is an open subgroup of $\prod_{\ell\not\in S}\Gamma_\ell$, it contains a subgroup of the form $\prod_{\ell\not\in S}V_\ell$,
where $V_\ell$ is an open subgroup of $\Gamma_\ell$ for all $\ell\not\in S$ and is equal to $\Gamma_\ell$ for all but finitely many $\ell$.
It suffices to prove that the image of $\prod_{\ell\not\in S}V_\ell$ in $\prod_{\ell\in S}U_\ell$ is finite.
Each factor $V_\ell$, $\ell\not\in S$ is virtually pro-$\ell$, so its image in any $U_\ell$, $\ell\in S$, is finite.  Since $U_\ell$ is torsion-free, each such image is trivial,
and it follows that the image of $\prod_{\ell\notin S}V_\ell$ in $\prod_{\ell\in S}U_\ell$ is trivial.
\end{proof}

Here are some properties for adelic/special adelic subgroups.

\begin{prop}
\label{Adelic-criterion}
A compact subgroup $\Gamma_{\A}\subset G_{\A}$ is adelic if and only if the following two conditions hold.
\begin{itemize}
\item[(i)] For each $\ell$ the projection of $\Gamma_{\A}$ to $G_\ell(\Q_\ell)$ has open image.
\item[(ii)] There exists a finite set $S$ of primes and for each $\ell\not\in S$ a maximal compact subgroup $\Gamma_\ell\subset G_\ell(\Q_\ell)$ such that
$$\prod_{\ell\not\in S} \Gamma_\ell \subset \pr_{\mathcal{L}\backslash S}(\Gamma_{\A}).$$
\end{itemize}
It is special adelic if and only if $\Gamma_\ell$ can be chosen hyperspecial for all $\ell\not\in S$.
\end{prop}

\begin{proof} If $\Gamma_{\A}\subset G_{\A}$ is adelic, then $\Gamma_{\A}$ contains an open subgroup of $\prod_\ell \Gamma_\ell$. We may assume the open subgroup is of the form 
$$\prod_{\ell\in S}\Gamma_\ell'\times \prod_{\ell\notin S}\Gamma_\ell,$$
where $S$ is a finite set of primes and $\Gamma_\ell'$ is an open subgroup of $\Gamma_\ell$. This implies (i) and (ii). 

Conversely, suppose (i) and (ii) hold.
Since $\Gamma_{\A}$ is compact, it is a closed subgroup of $\prod_\ell\Gamma_\ell$ of finite index by Lemma \ref{One-to-finite} and Lemma \ref{compare}.

The proof for the special adelic case is similar.
\end{proof}

\begin{prop}
\label{Nested}
If $G^1_{\A}\subset G^2_{\A}$ are both adelic subgroups of $G_{\A}$, then $G^1_{\A}$ is of finite index in $G^2_{\A}$.
\end{prop}

\begin{proof}
This follows easily from Proposition~\ref{Adelic-criterion}, Lemma~\ref{One-to-finite}, and the compactness of $G^1_{\A}$ and  $G^2_{\A}$.
\end{proof}

For any semisimple $G$ over a field $F$, we denote by $G^{\ad}$ the adjoint quotient of $G$.  The conjugation action of $G$ on itself factors through $G^{\ad}$ at the algebraic group level, so $G^{\ad}(F)$ acts on $G(F)$.  We can characterize special adelic subgroups of $\prod_\ell G(\Q_\ell)$ in terms of this action.

\begin{prop}
If $G$ is a simply connected semisimple group over $\Q$ and $G_\ell = G\times\Q_\ell$ for all $\ell$, then a compact subgroup $\Gamma_{\A}\subset G_{\A}$ is special adelic if and only if
it is conjugate under the action of $G^{\ad}_{\A}:=\prod_\ell G^{\ad}(\Q_\ell)$ on $G_{\A}$ to an open subgroup of $G(\A_f)\subset G_{\A}$.
\end{prop}

\begin{proof}
Let $G = \Spec \Q[x_1,\ldots,x_m]/(f_1,\ldots,f_n)$.  Then 
$$f_i\in \Z[1/N][x_1,\ldots,x_n]$$
for some $N$, and 
$$\cG := \Spec \Z[1/N][x_1,\ldots,x_m]/(f_1,\ldots,f_n)$$
is an affine scheme with generic fiber $G$.  Replacing $N$ with a suitable positive integral multiple, we may assume first that $\cG$ is a group scheme,
next that $\cG$ is smooth,
and finally that it has reductive fibers \cite[XIX~2.6]{SGA3}.
We have $G(\A_f) = \cG(\A_f)$.  
A subset $X$ of $G(\A_f)$ is a neighborhood of $(a_1,\ldots,a_m)\in G(\A_f)\subset \A_f^m$ if and only if
there exist open sets $U_\ell\subset \Q_\ell$ such that $U_\ell = \Z_\ell$ for all $\ell\gg 1$ 
and such that if $(b_i)_\ell - (a_i)_\ell \in U_\ell$ for all $i$ and $\ell$
and $(b_1,\ldots,b_m)\in G(\A_f)$, then $(b_1,\ldots,b_m)\in X$.  In particular, if $\ell$ is large enough that all $(a_i)_\ell\in \Z_\ell$, $U_\ell=\Z_\ell$, and $\ell\nmid N$,
then the condition on $((b_1)_\ell,\ldots,(b_m)_\ell)$ is just that it lies in $\cG(\Z_\ell)$. 
Since $\cG(\Z_\ell)\subset G(\Q_\ell)$ is hyperspecial maximal compact for all $\ell\gg1$ and all hyperspecial maximal compact subgroups of $G(\Q_\ell)$ are $G^{\ad}(\Q_\ell)$-conjugate \cite[p.~47]{Tits}, any special adelic $\Gamma_{\A}$ is conjugate in $G^{\ad}_{\A}$ to a neighborhood of the identity in $G(\A_f)$.

Conversely, embed $G$ into $\GL_r$. Then every compact open subgroup $\Gamma_{\A}$ of $G(\A_f)$ contains an open subgroup $\prod_\ell \Gamma_\ell$ 
of $G(\A_f)\cap\GL_r(\hat{\Z})$, where $\Gamma_\ell$ is open in $G(\Q_\ell)$.
For all $\ell\gg 1$, $\cG$ is smooth and reductive over $\Z_\ell$, and it follows that $\Gamma_\ell=\cG(\Z_\ell)$ is a hyperspecial maximal compact subgroup of
$G(\Q_\ell) = \cG(\Q_\ell)$.  Thus, $\Gamma_{\A}$ and therefore all $G^{\ad}_{\A}$-conjugates of $\Gamma_{\A}$ in $G_{\A}$ are special adelic.
\end{proof}

\begin{prop}
If $G$ is a simply connected, absolutely almost simple algebraic group over $\Q$ and $\Gamma$ is a finitely generated, Zariski dense subgroup 
of $G(\Q)$ which is relatively compact in $G(\Q_\ell)$ for all $\ell$, then the closure $\bar\Gamma$ of $\Gamma$ in $\prod_\ell G(\Q_\ell)$ is special adelic.
\end{prop}

\begin{proof}
We can regard $G$ as the generic fiber of an affine group scheme $\cG$ defined over $\Z[1/N]$ for some $N$.
By Matthews, Vaserstein, and Weisfeiler \cite{MVW}, if $\ell$ is sufficiently large, the closure of $\Gamma$ in $G(\Q_\ell)$
with respect to the $\ell$-adic topology equals $\cG(\Z_\ell)\subset \cG(\Q_\ell) = G(\Q_\ell)$.  Moreover,
for sufficiently large $M$,  the closure of $\Gamma$ in $\prod_{\ell>M} \cG(\Z_\ell)$ is open.  Defining $\Gamma_\ell$ to be the closure of the image of $\Gamma$ in $G(\Q_\ell)$,
we see that all but finitely many $\Gamma_\ell$ are hyperspecial.  Moreover, for all $\ell$, $\Gamma_\ell$ is a compact Zariski-dense subgroup of $G(\Q_\ell)$, and therefore
by a classical theorem of Chevalley, it is open in $G(\Q_\ell)$.  The proposition follows immediately by Proposition \ref{Adelic-criterion}.
\end{proof}

\begin{prop}\label{coro}
Together, the Mumford-Tate conjecture and Conjecture~\ref{main} imply Conjecture~\ref{MT}.
\end{prop}

\begin{proof} Let $U$ be a subgroup of $H^{\sc}(\A_f)$ which is a compact, special adelic subgroup of $\prod_\ell H^{\sc}(\Q_\ell)$. It suffices to show that $U$ contains an open subgroup of $H^{\sc}(\A_f)$. Embed $H^{\sc}$ in some $\GL_m$. 
Since $U$ is special adelic, it contains $\prod_{\ell\in S}\Gamma_\ell'\times \prod_{\ell\notin S}\Gamma_\ell$ (for a finite subset $S$ of $\mathcal{L}$) such that $\Gamma_\ell'$ 
is open and of finite index in $H^{\sc}(\Q_\ell)\cap\GL_m(\Z_\ell)$ and $\Gamma_\ell$ is a hyperspecial maximal compact subgroup of $H^{\sc}(\Q_\ell)$.  
It suffices to show that $\prod_{\ell\in S}\Gamma_\ell'\times \prod_{\ell\notin S}\Gamma_\ell$ is open in $H^{\sc}(\A_f)$. Since 
$$\prod_{\ell\in S}\Gamma_\ell'\times \prod_{\ell\notin S}\Gamma_\ell\subset H^{\sc}(\A_f)$$ 
is a direct product, we assume $\Gamma_\ell$ is a subgroup of $\GL_m(\Z_\ell)$ for $\ell\notin S$. Since 
$\Gamma_\ell$ is maximal compact,  $\Gamma_\ell=H^{\sc}(\Q_\ell)\cap \GL_m(\Z_\ell)$ and we conclude that 
$$\prod_{\ell\in S}\Gamma_\ell'\times \prod_{\ell\notin S}\Gamma_\ell\subset H^{\sc}(\A_f)\cap \GL_m(\widehat{\Z})$$
is of finite index. Since $H^{\sc}(\A_f)\cap \GL_m(\widehat{\Z})$ is open in $H^{\sc}(\A_f)$, we are done.
\end{proof}

\section{Finite products of commutators}

Let $G$ be a connected reductive algebraic group over a field $F$ of characteristic $0$.  Let $Z$ denote the 
identity component of the center of $G$ and $G^{\sc}$ the universal covering group of the derived group $G^{\der}$
of $G = G^{\der} Z$.  The covering map $\pi\colon G^{\sc}\to G^{\der}$ defines an isogeny $G^{\sc}\times Z\to G$ which is separable and therefore central \cite[22.3]{Borel}.
By the definition of \emph{central} (\cite[2.2]{Borel-Tits}), the commutator morphism on $G^{\sc}\times Z$ factors through a morphism $G\times G\to G^{\sc}\times Z$.  Composing with  projection
onto the first factor, we obtain a morphism $\kappa\colon G\times G\to G^{\sc}$ which is defined so that if $x_1,x_2\in G^{\sc}(\bar F)$ and $z_1,z_2\in Z(\bar F)$, then 
$$\kappa(\pi(x_1) z_1,\pi(x_2) z_2) = [x_1,x_2].$$

Given a subgroup $\Gamma\subset G(F)$, there is a commutative diagram
\begin{equation*}
\label{square}
\xymatrix{[\Gamma,\Gamma]\ar@{^{(}->}[d]\ar[r] &G^{\sc}(F)\ar[d] \\
		\Gamma\ar@{^{(}->}[r] & G(F).\\
}
\end{equation*}

\begin{prop}\label{3.1}
If $\Gamma_\ell\subset G_\ell(\Q_\ell)$ is a compact subgroup, then there exists $k$ 
such that every element in the subgroup $\Delta_\ell\subset G_\ell^{\sc}(\Q_\ell)$ generated by 
$\kappa_\ell(\Gamma_\ell,\Gamma_\ell)$ is a product of $k$ elements in this set.
Moreover, if the image of $\Gamma_\ell$ in the adjoint quotient $G_\ell^{\ad}$ of $G$ is Zariski-dense, then 
$\Delta_\ell$ is a compact open subgroup of $G_\ell^{\sc}(\Q_\ell)$.
\end{prop}

\begin{proof}

Consider the isogeny $\pi\colon Z_\ell\times G_\ell^{\sc}\to G_\ell$.  Let $E_\lambda$ be a finite extension of $\Q_\ell$ containing the Galois closure of
every extension of $\Q_\ell$ of degree $\le \deg\pi$.  Thus, for $x\in G_\ell(\Q_\ell)$, every point of $\pi^{-1}(x)$ is defined over $E_\lambda$.
Let 
$$\tilde \Gamma_\ell := \pi^{-1}(\Gamma_\ell)\subset (Z_\ell\times G_\ell^{\sc})(E_\lambda).$$
As $\pi$ is finite, it is projective, so $(Z_\ell\times G_\ell^{\sc})(E_\lambda)\to G_\ell(E_\lambda)$ is a proper map,
and it follows that $\tilde \Gamma_\ell\to \Gamma_\ell$ is proper and therefore that $\tilde \Gamma_\ell$ is compact.  By definition, 
$\Delta_\ell$ is the group generated by $[\tilde \Gamma_\ell,\tilde \Gamma_\ell]$.
By a theorem of Jaikin-Zapirain \cite[1.3]{JZ},
there exists $k$, such that every element of $\Delta_\ell$ is a product of $k$ elements of  $[\tilde \Gamma_\ell,\tilde \Gamma_\ell]$. 

Now, $\kappa_\ell$ factors through $G_\ell^{\ad}\times G_\ell^{\ad}$.  If the image of $\Gamma_\ell$ in $G_\ell^{\ad}(\Q_\ell)$ is Zariski-dense,
then $\Delta_\ell$ is Zariski-dense and compact and therefore open by a theorem of Chevalley (see, e.g., \cite{Pink}).
\end{proof}

The following result can be looked at as a weak $\ell$-adic analogue of Got\^o's theorem \cite{Goto} that every element of a compact semisimple real Lie group is a commutator.

\begin{prop}
Let $G_\ell$ be a semisimple group over $\Q_\ell$ and $\Gamma_\ell$ an open subgroup of $G_\ell(\Q_\ell)$.  Then
$$\{xyx^{-1}y^{-1}\mid x,y\in \Gamma_\ell\}$$
has non-empty interior.
\end{prop}

\begin{proof}
First we claim that if $G$ is any semisimple algebraic group over a field of characteristic zero, the commutator morphism $G\times G\to G$
is dominant.  This can be deduced from Got\^o's theorem or from Borel's theorem \cite{Borel-Word} that any non-trivial word in a free group on $n$ letters
defines a dominant morphism $G^n\to G$.  Every dominant morphism of varieties in characteristic zero is generically smooth, so there exists a proper Zariski-closed subset $X\subset G_\ell\times G_\ell$ such that the commutator map $G_\ell^2\to G_\ell$ is smooth outside $X$.
Now, the interior of $X(\Q_\ell) \subset G_\ell^2(\Q_\ell)$ is trivial, so $\Gamma_\ell^2$ has some point $(x,y)$ which is a smooth point for
the commutator map, and it follows from the $\ell$-adic implicit function theorem, that $xyx^{-1}y^{-1}$ is an interior point of
the set of commutators of $\Gamma_\ell$.
\end{proof}

\begin{cor}
\label{goto}
Let $G_\ell$ be a semisimple group over $\Q_\ell$ and $\Gamma_\ell$ an open subgroup of $G_\ell(\Q_\ell)$.  Then
there exists an open subgroup of $G_\ell(\Q_\ell)$ in which every element is a product of two commutators of elements of $\Gamma_\ell$.
\end{cor}

\begin{proof}
If $z$ is a commutator, then $z^{-1}$ is a commutator as well.  If a neighborhood of $z$ consists of commutators, then
a neighborhood of the identity consists of products of two commutators.
\end{proof}

The following theorem is a  variant of a result of Avni-Gelander-Kassabov-Shalev \cite[Theorem~2.3]{AGKS}.

\begin{thm}
\label{Avni}
Let $G_\ell$ be a simple connected semisimple algebraic group over $\Q_\ell$, with $\ell  > \max(5,\dim G_\ell)$.
If $\Gamma_\ell$ is a hyperspecial maximal compact subgroup of $G_\ell(\Q_\ell)$,  then 
every element in $\Gamma_\ell$ is a product of two commutators of elements of $\Gamma_\ell$. In particular, $\Gamma_\ell$ is perfect.
\end{thm}

\begin{proof}
As $\Gamma_\ell$ is hyperspecial, it is of the form $\cG(\Z_\ell)$, where $\cG$ is a smooth affine group scheme
over $\Z_\ell$ with connected, simply connected, semisimple fibers.  In particular, as $\ell\ge 5$, $\cG(\F_\ell)$ is a product of quasisimple groups.

Let $\g$ denote the Lie algebra of $\cG_{\F_\ell}$, the
closed fiber of $\cG$.  Let $\Ad$ denote the adjoint representation of $\cG(\F_\ell)$ acting on $\g$.
We consider the commutator map $(x,y)\mapsto [x,y]=xyx^{-1}y^{-1}$ as a morphism $\cG^2\to \cG$.
Let $\bar x,\bar y\in \cG(\F_\ell)$.  Identifying the tangent space $T_{(\bar x,\bar y)} \cG_{\F_\ell}^2$
(resp.\ $T_{[\bar x,\bar y]}\cG_{\F_\ell}$)  with  $\g^2$ (resp.\ $\g$) 
via right-translation by 
$(\bar x,\bar y)$ (resp.\ $[\bar x,\bar y]$), the map on tangent spaces induced by the commutator map at $(\bar x,\bar y)$
is given by
$$(X,Y)\mapsto (\Ad(\bar x)-\Ad(\bar x\bar y)) X  + (\Ad(\bar x\bar y) - \Ad(\bar x\bar y\bar x^{-1})) Y.$$
If
\begin{equation}
\label{smooth}
(1- \Ad(\bar y))\g + \Ad(\bar y)(1-\Ad(\bar x^{-1}))\g  = \g,
\end{equation}
then by Hensel's lemma, every $z\in \cG(\Z_\ell)$ which is congruent to $[x,y]$ (mod $\ell$) is of the form $[x',y']$
for $x'$ and $y'$ congruent to $x$ and $y$ respectively. Since every element of $\cG(\F_\ell)$ ($\ell\geq 5$) is a commutator \cite{LOST}, it follows that every element in $\cG(\Z_\ell)$ is the product of
two commutators.

The hypothesis on $\ell$ guarantees that the adjoint representation
$\g$ is a semisimple representation  whose irreducible factors $\g_i$ are the simple factors of $\g$, which correspond to the almost simple factors
$G_i$ of $\cG(\F_\ell)$ \cite[Proposition~3.3, Cor.~3.7.1]{Vasiu}. Since $\ell>5$, there is a non-degenerate pairing $\langle\,,\,\rangle$ on $\g$ (see \cite{BZ}) such that if $a,b,c\in\g$, then
$$\langle [a,b],c \rangle=\langle a,[b,c]\rangle.$$
Then (\ref{smooth}) fails if and only if there exists $a\in\g$ non-zero such that
$$\langle a,\Ad(\bar y)(1-\Ad(\bar x)) b\rangle = \langle a,(1-\Ad(\bar y)) c\rangle = 0$$
for all $b,c\in\g$ or, equivalently,
$$
\langle a,\Ad(\bar y) b\rangle = \langle a,\Ad(\bar y)\Ad(\bar x) b\rangle,\ \langle a,c\rangle = \langle a,\Ad(\bar y)c\rangle,
$$
or, again,
\begin{equation}
\label{switch}
\langle \Ad(\bar y)^{-1} a,b\rangle = \langle \Ad(\bar x)^{-1}\Ad(\bar y)^{-1} a,b\rangle,\ \langle a,c\rangle = \langle \Ad(\bar y)^{-1} a,c\rangle
\end{equation}%
for all $b,c\in\g$. 
By the non-degeneracy of the pairing, the second condition in (\ref{switch}) implies $\Ad(\bar y)^{-1}$ fixes $a$, and the first condition now
implies that $\Ad(\bar x)^{-1}$ fixes $a$ as well.
In particular, if $a$ has a non-zero component $a_i$ in some simple factor $\g_i$ of $\g$, then the images of $\bar x$ and $\bar y$ in the
corresponding quasi-simple factor $G_i(\F_\ell)$ lie in the stabilizer of $a_i$, which is a proper subgroup of $G_i(\F_\ell)$ by the irreducibility of the
adjoint representation.

It is well known that every finite simple group can be generated by two elements \cite{AG}, and 
it follows that the same is true for every perfect central extension of a finite simple group.
Applying these results to the quasi-simple factors of $\cG(\F_\ell)$, we obtain elements $\bar x$ and $\bar y$ 
whose projection to each quasisimple factor $G_i(\F_\ell)$ generates $G_i(\F_\ell)$.  
Thus, we have (\ref{smooth}), and the proposition follows.
\end{proof}

\begin{cor}
\label{Bounded-dim}
Let $\Gamma_{\A}$ denote a special adelic subgroup of $\prod_\ell G_\ell(\Q_\ell)$, where $\dim G_\ell$ is bounded over all $\ell$.
Then there exists a special adelic subgroup $\Gamma'_{\A}\subset \Gamma_{\A}$ such that every element
of $\Gamma'_{\A}$ is a product of two commutators of elements of $\Gamma_{\A}$.
\end{cor}

\begin{proof}
This is an immediate consequence of Corollary~\ref{goto} and Theorem~\ref{Avni}.
\end{proof}

Let $\Gamma_\ell\subset G_\ell(\Q_\ell)$ be a compact subgroup, where $G_\ell/\Q_\ell$ is connected reductive.
Denote by $\Gamma_\ell^{\ss}$ the image of $\Gamma_\ell$ under $G_\ell(\Q_\ell)\to G_\ell/Z_\ell(\Q_\ell)$, where $Z_\ell$ is the identity component of the center of $G_\ell$. Denote by $\Gamma_\ell^{\sc}$ the pre-image of $\Gamma_\ell^{\ss}$ under $G_\ell^{\sc}(\Q_\ell)\to G_\ell/Z_\ell(\Q_\ell)$. These constructions appear in \cite{HL},\cite{HL15}.

\begin{cor}\label{equal}
The subgroup $U_\ell$ of $G_\ell^{\sc}(\Q_\ell)$ generated by $\kappa_\ell(\Gamma_\ell,\Gamma_\ell)$ is contained in $\Gamma_\ell^{\sc}$.
If $\Gamma_\ell^{\sc}\subset G_\ell^{\sc}(\Q_\ell)$ is hyperspecial maximal compact and $\ell>\max\{5,\dim G_\ell^{\sc}\}$, 
then $\Gamma_\ell^{\sc}\subset \kappa_\ell(\Gamma_\ell,\Gamma_\ell)^2$. Hence, $\Gamma_\ell^{\sc}=U_\ell$.
\end{cor}

\begin{proof}
Let $\pi_\ell:G_\ell^{\sc}\to G_\ell^{\der}$ be the covering isogeny.  By the definition of $\kappa_\ell$, we obtain 
$$\pi_\ell(U_\ell)\subset\Gamma_\ell\cap G_\ell^{\der}(\Q_\ell)\subset\Gamma_\ell.$$ 
This implies $U_\ell\subset \Gamma_\ell^{\sc}$ by the definition of $\Gamma_\ell^{\sc}$.
If $\Gamma_\ell^{\sc}\subset G_\ell^{\sc}(\Q_\ell)$ is hyperspecial maximal compact and $\ell>\max\{5,\dim G_\ell^{\sc}\}$,
then every element of $\Gamma_\ell^{\sc}$ is a product of two commutators by Theorem \ref{Avni}.
Since every commutator of $\Gamma_\ell^{\sc}$ belongs to $\kappa_\ell(\Gamma_\ell,\Gamma_\ell)$ by the definitions of $\Gamma_\ell^{\sc}$ and $\kappa_\ell$, we obtain
$$\Gamma_\ell^{\sc}\subset\kappa_\ell(\Gamma_\ell,\Gamma_\ell)^2\subset U_\ell.$$
Hence, $\Gamma_\ell^{\sc}=U_\ell$.
\end{proof}

\begin{remark}\label{rem-con}
Suppose $\Gamma_\ell$ and $G_\ell$ are respectively the image and the algebraic monodromy group of the representation $V_\ell$ in $\mathsection1$. The second author has conjectured that $\Gamma_\ell^{\sc}\subset G_\ell^{\sc}(\Q_\ell)$ is hyperspecial maximal compact for $\ell\gg1$ \cite{Larsen}.
By Corollary \ref{equal}, this implies $\kappa_\ell(\Gamma_\ell,\Gamma_\ell)$ generates a hyperspecial maximal compact subgroup of $ G_\ell^{\sc}(\Q_\ell)$ for $\ell\gg1$.
\end{remark}

\begin{thm}
\label{Generation}
Let $\Gamma_{\A}\subset \prod_\ell G_\ell(\Q_\ell)$ be a compact subgroup, where $\dim G_\ell^{\sc}$ is bounded by $N$.  
Suppose that for all $\ell$, $\pr_\ell(\Gamma_{\A})$ is Zariski-dense in $G_\ell$.
Suppose further that there exist a finite set of primes $S$ and a collection of compact open subgroups $\Gamma_\ell\subset G_\ell(\Q_\ell)$
for all $\ell$ such that 
\begin{enumerate}
\item[(i)] $\prod_{\ell}\Gamma_\ell\subset \Gamma_{\A};$
\item[(ii)] $\Gamma_\ell^{\sc}\subset G_\ell^{\sc}(\Q_\ell)$ is hyperspecial maximal compact for all $\ell\notin S$.
\end{enumerate}  
Then $\kappa(\Gamma_{\A},\Gamma_{\A})$ generates a special adelic subgroup
of $\prod_\ell G_\ell^{\sc}(\Q_\ell)$, which is equal to $\kappa(\Gamma_{\A},\Gamma_{\A})^k$ for some $k$. Moreover, $\kappa(\Gamma_{\A},\Gamma_{\A})^{2}$ contains a special adelic subgroup of $\prod_\ell G_\ell^{\sc}(\Q_\ell)$.
\end{thm}

\begin{proof} Since $\pr_\ell(\Gamma_{\A})$ is compact and Zariski-dense in $G_\ell$, 
$\kappa_\ell(\pr_\ell(\Gamma_{\A}),\pr_\ell(\Gamma_{\A}))$ generates 
a compact and open subgroup $U_\ell$ of $G_\ell^{\sc}(\Q_\ell)$ for all $\ell$  by Proposition \ref{3.1}. 
Hence, $\kappa(\Gamma_{\A},\Gamma_{\A})$ generates a subgroup $U_\A$ of $\prod_\ell U_\ell$, a direct product of compact groups.
Since (i) implies
$$\kappa(\prod_\ell\Gamma_\ell,\prod_\ell\Gamma_\ell)\subset \kappa(\Gamma_\A,\Gamma_\A)\subset \prod_\ell U_\ell,$$
it suffices to prove the theorem  assuming $\Gamma_\A=\prod_\ell\Gamma_\ell$.
Hence, we obtain for all $k\in\mathbb{N}$ that
\begin{equation}\label{product-eqt}
\prod_\ell\kappa_\ell(\Gamma_\ell,\Gamma_\ell)^k=\kappa(\Gamma_\A,\Gamma_\A)^k.
\end{equation}

Suppose $S$ is large enough to contain the primes that are not greater than $\mathrm{max}\{5,N\}$.
If $\ell\notin S$, then $\Gamma_\ell^{\sc}$ is a hyperspecial maximal compact subgroup of $G_\ell^{\sc}(\Q_\ell)$ by (ii).
By Corollary \ref{equal}, 
 $\Gamma_\ell^{\sc}\subset\kappa_\ell(\Gamma_\ell,\Gamma_\ell)^2$ for $\ell\notin S$.
Since $S$ is finite, by Proposition \ref{3.1}, there exists $k\ge 2$ such that $\kappa_\ell(\Gamma_\ell,\Gamma_\ell)^k$ is a
compact open subgroup $\Delta_\ell$ of $G_\ell^{\sc}(\Q_\ell)$ for all $\ell\in S$.
We conclude by (\ref{product-eqt}) that 
$$U_\A:=\kappa(\Gamma_\A,\Gamma_\A)^k=\prod_{\ell\in S}\Delta_\ell\times\prod_{\ell\notin S}\Gamma_\ell^{\sc}$$ 
is special adelic.

By Corollary \ref{Bounded-dim}, there exists a special adelic subgroup $U_\A'\subset U_\A$ such that every element 
 $x\in U_\A'$ is a product of two commutators of $U_\A$, i.e, there exist $y,z,y',z'\in U_\A$ such that
$$x=[y,z][y',z'].$$
Since $\pi_\A:\prod_\ell G_\ell^{\sc}(\Q_\ell)\to \prod_\ell G_\ell^{\der}(\Q_\ell)$ maps $U_\A$ into $\Gamma_\A\cap \prod_\ell G_\ell^{\der}(\Q_\ell)$, $x=[y,z][y',z']\in\kappa(\Gamma_\A,\Gamma_\A)^2$ by the definition of $\kappa$.
Therefore, $\kappa(\Gamma_\A,\Gamma_\A)^2$ contains the special adelic subgroup $U_\A'$.
\end{proof}

\begin{remark}
Let $\Gamma_\A:=\Phi_\A(G_K)\subset\prod_\ell G_\ell(\Q_\ell)$ be the adelic image.
Then (i) in Theorem \ref{Generation} always holds by Serre \cite{Serre-AI} while 
 (ii) is Larsen's conjecture \cite{Larsen} (see Remark \ref{rem-con}).
Hence, Larsen's conjecture implies Conjecture \ref{main} by Theorem \ref{Generation}.
\end{remark}

\section{Galois representations}
Recall the assumption that the algebraic monodromy group $G_\ell$ of our
$\ell$-adic representation $\Phi_\ell:\Gal_K\to\Aut(V_\ell)\cong\GL_n(\Q_\ell)$ is connected reductive for all $\ell$.

\begin{thm}
With notations as in the introduction, Conjecture~\ref{main} holds if either of the following statements holds:
\begin{itemize}
\item[(a)] For all $\ell$, $G^{\sc}_\ell$ is a product of type A simple factors;
\item[(b)] For some $\ell$, $G^{\sc}_\ell\times_{\Q_\ell}\bar\Q_\ell$ is a product of simple factors $\SL_{r+1}$ of rank $r=4$, $r=6$, or $r\ge9$ with at most one rank $4$ factor.
\end{itemize}
\end{thm}

\begin{proof} Let $\Gamma_{\A}:=\Phi_{\A}(G_K)\subset \prod_{\ell}G_\ell(\Q_\ell)$.   The dimension of $G_\ell^{\sc}$ is bounded in terms of the dimension of $V_\ell$ which is independent of $\ell$, and $\pr_\ell(\Gamma_{\A})$ is compact and Zariski dense in $G_\ell$. It suffices to verify the conditions (i) and (ii) of Theorem~\ref{Generation} for a finite index subgroup of $\Gamma_{\A}$. The fields defined by $\ker\Phi_\ell$ are almost linearly disjoint by Serre \cite{Serre-AI}, so there exists a finite extension $K'$ of $K$
such that
$$\prod_{\ell}\Gamma_\ell:=\prod_\ell\Phi_\ell(G_{K'})=\Phi_\A(G_{K'})\subset \Gamma_\A.$$
Therefore, Theorem \ref{Generation}(i) always holds for $\Gamma_\A$.

Since  (b) implies  (a) by \cite[Theorem 3.21]{Hu}, the main result of \cite{HL} implies 
$\Gamma_\ell^{\sc}$ is a hyperspecial maximal compact Lie subgroups  of $G_\ell^{\sc}(\Q_\ell)$ for all sufficiently large $\ell$.
Hence, the condition \ref{Generation}(ii) holds for (a) and (b).
 \end{proof}

Conjecture~\ref{main} also holds for abelian varieties, as we will prove in a later paper.

\section{Connections with maximal motives}
We follow closely \cite{Serre-GM} in this section. Assume the standard conjectures of algebraic cycles \cite{Groth,Kle} and the Hodge conjecture, denote the category of (pure) motives over number field $K$ by $\mathcal{M}$. The category $\mathcal{M}$ is a semisimple neutral Tannakian category over $\Q$. Let $E\in\mathrm{ob}(\mathcal{M})$ be a motive, $\mathcal{M}(E)$ is the smallest full Tannakian subcategory of $\mathcal{M}$ containing $E$. If $E'$ is an object of $\mathcal{M}(E)$, then we say that $E'$ is dominated by $E$, denoted by $E'\prec E$. Fix an embedding $\sigma:K\rightarrow\C$, there exists a fiber functor from $\mathcal{M}$ to the category of finite dimensional $\Q$-vector space
\begin{equation*}
h_\sigma:\mathcal{M}\rightarrow \mathrm{Vect}_{\Q},
\end{equation*}
which is an exact faithful $\Q$-linear tensor functor. The scheme of automorphisms $G_{\mathcal{M}}$ of $h_\sigma$ is called the motivic Galois group over $K$. It is a projective limit of $\Q$-reductive groups $G_{\mathcal{M}(E)}$ relative to dominance of motives $E$, where $G_{\mathcal{M}(E)}$ is the scheme of automorphisms of the restriction of $h_\sigma$ to $\mathcal{M}(E)$. The category $\mathcal{M}$ is equivalent to $\mathrm{Rep}_\Q G_{\mathcal{M}}$, the category of finite dimensional $\Q$-linear representation of $G_{\mathcal{M}}$.

Suppose $E$ is a motive over $K$. The $\ell$-adic cohomology of $E$ (over $\bar{K}$) induces an $\ell$-adic Galois representation
\begin{equation*}
\Phi_{\ell,E}:G_K\rightarrow \mathrm{GL}(h_\sigma(E)\otimes\Q_\ell)
\end{equation*}
and the $\ell$-adic image is contained in $G_{\mathcal{M}(E)}(\Q_\ell)$. 
Let $L$ be an lattice of $h_\sigma(E)$.
Since $L\otimes \Z_\ell$ is stable under the action of $\Phi_{\ell,E}$ for almost all primes $\ell$ \cite[$\mathsection10$]{Serre-GM}, we obtain by consolidating representations $\Phi_{\ell,E}$ for different $\ell$ an adelic representation
\begin{equation*}
\Phi_E: G_K\rightarrow G_{\mathcal{M}(E)}(\A_f).
\end{equation*}

Suppose $G_{\mathcal{M}(E)}$ is connected, then $G_{\mathcal{M}}^{\circ}\rightarrow G_{\mathcal{M}(E)}$ is surjective. The motive $E$ is said to be \textit{maximal} if whenever $G'$ is connected reductive and $G'\rightarrow G_{\mathcal{M}(E)}$ is a non-trivial isogeny, then the homomorphism $G_{\mathcal{M}}^\circ \rightarrow G_{\mathcal{M}(E)}$ does not factor through $G'\rightarrow G_{\mathcal{M}(E)}$ \cite[11.2]{Serre-GM}. Serre has conjectured the following.

\begin{conj}\label{maximal}\cite[11.4]{Serre-GM}
Suppose $G_{\mathcal{M}(E)}$ is connected. The following two properties are equivalent:
\begin{itemize}
\item[(i)] $E$ is maximal.
\item[(ii)] $\mathrm{Im}(\Phi_E)$ is open in the adelic group $G_{\mathcal{M}(E)}(\A_f)$.
\end{itemize}
\end{conj}

\begin{conj}\label{dominant}\cite[11.8]{Serre-GM}
By taking a finite extension of $K$, the motive $E$ is dominant by a maximal motive $E'$ such that 
$G_{\mathcal{M}(E')}\to G_{\mathcal{M}(E)}$
is a connected covering.
\end{conj}

\begin{thm}\label{connection}
Conjecture \ref{maximal}, Conjecture \ref{dominant}, and the Mumford-Tate conjecture imply Conjecture \ref{MT}.
\end{thm}

\begin{proof}
Suppose $E$ is a motive such that the $\ell$-adic realization of $E$ 
is the $i$th \'etale cohomology of a non-singular projective variety $X$ over $K$.
Suppose the algebraic monodromy group $G_\ell$ of the representation $V_\ell$ is connected for all $\ell$.
The Mumford-Tate group $H$ is a subgroup of $G_{\mathcal{M}(E)}$ by \cite[3.3]{Serre-GM}.
Conjecture \ref{maximal} and Conjecture \ref{dominant} imply 
$G_{\mathcal{M}(E)}\times\Q_\ell=G_\ell$.
The Mumford-Tate conjecture implies 
$H\times\Q_\ell=G_\ell$.
Hence, we obtain $G_{\mathcal{M}(E)}=H$.

Let $\Gamma_\A$ be $\Phi_E(G_K)$. It is a compact subgroup of $\prod_\ell H^{\sc}(\Q_\ell)$.
By Proposition \ref{coro}, it suffices to show that $\kappa(\Gamma_\A,\Gamma_\A)^k$ is a 
special adelic subgroup of $\prod_\ell H^{\sc}(\Q_\ell)$ for some $k$ and contains
such a subgroup if $k\geq 2$.
Suppose $K$ is large enough  that $E$ is dominated by a maximal motive $E'/K$ by Conjecture \ref{dominant}.
Let $\Gamma_\A'$ be $\Phi_{E'}(G_K)$. Since $G_{\mathcal{M}(E')}\to G_{\mathcal{M}(E)}=H$
is a connected covering, $\kappa(\Gamma_\A,\Gamma_\A)=\kappa(\Gamma_\A',\Gamma_\A')$ in $H^{\sc}(\A_f)$.
Without loss of generality, we may assume $E$ is maximal. By Conjecture \ref{maximal}, $\Gamma_\A\subset H(\A_f)$ is an open subgroup.
Therefore, there exist a finite set $S$ of primes and compact open subgroups $\Gamma_\ell\subset H(\Q_\ell)$ for all $\ell\in S$
such that the following hold.
\begin{enumerate}
\item[(i)] $\prod_{\ell\in S}\Gamma_\ell\prod_{\ell\notin S} H(\Z_\ell)\subset \Gamma_\A\subset\prod_\ell H(\Q_\ell)$.
\item[(ii)] $\forall\ell\notin S$, the morphism $H^{\sc}(\Z_\ell)\to H(\Z_\ell)$ is well defined and $H^{\sc}(\Z_\ell)$ is hyperspecial maximal compact in $H^{\sc}(\Q_\ell)$. Hence, $H(\Z_\ell)^{\sc}=H^{\sc}(\Z_\ell)$.
\end{enumerate}
We are done by Theorem \ref{Generation}.
\end{proof}

\end{document}